\documentclass[12pt]{amsart}
\usepackage{graphicx} 
\usepackage{amssymb, enumerate}
\newcommand\Z{\mathbb{Z}}
\usepackage{parskip}
\usepackage{geometry}
\newcommand\Q{\mathbb{Q}}
\newcommand\C{\mathbb{C}}
\newcommand\Oo{\mathcal{O}}
\DeclareMathOperator{\Li}{Li}
\DeclareMathOperator{\disc}{disc}
\renewcommand\supset\supseteq
\newtheorem{prop}{Proposition}
\newtheorem{lem}[prop]{Lemma}
\newtheorem{thm}[prop]{Theorem}

\begin{document}

\title{Orders in the Euclidean court}

\author{Carlo Francisco Adajar}

     \address{Department of Mathematics \\ University of Georgia\\ Athens, GA 30602}
     \email{paco@uga.edu}

\author{Kevin J. McGown}
\address{Department of Mathematics and Statistics \\ California State University, Chico\\ Chico, CA 95929}
     \email{kmcgown@csuchico.edu}

\author{Paul Pollack}

     \address{Department of Mathematics \\ University of Georgia\\ Athens, GA 30602}
     \email{pollack@uga.edu}

\author{Enrique Treviño}
     \address{Department of Mathematics and Computer Science\\ Lake Forest College\\ Lake Forest, IL 60045}
     \email{trevino@lakeforest.edu}
\begin{abstract} Following Johnson, Queen, and Sevilla, we call an order $\mathcal{O}$ in a number field $K$ \textsf{generalized Euclidean} if for all $\alpha,\beta \in \Oo$ with $\beta\ne 0$ for which $\alpha\Oo+\beta\Oo$ is a principal ideal of $\Oo$, there are $\gamma, \rho \in \Oo$ with $\alpha =\beta \gamma + \rho$ and $|N_{K/\Q}(\rho)| < |N_{K/\Q}(\beta)|$. We prove that only finitely many quadratic orders are generalized Euclidean. Additionally, assuming GRH, we show that for each fixed $d > 1$ there are only finitely many integers $m$ for which $T^d-m$ is irreducible over $\Q$ and $\Z[m^{1/d}] = \Z[T]/(T^d-m)$ is generalized Euclidean. When $d$ is even, we can remove the assumption of GRH.
\end{abstract}

\maketitle

\section{Introduction}
Let $K$ be a number field with ring of integers $\Oo_K$. One calls $\Oo_K$  \textsf{norm-Euclidean} if, for every pair of $\alpha, \beta \in \Oo_K$ with $\beta\ne 0$, there are $\gamma, \rho \in \Oo_K$ for which
\begin{equation}\label{E:euclid}
\alpha = \beta \gamma + \rho
\quad\text{and}\quad |N_{K/\Q}(\rho)| < |N_{K/\Q}(\beta)|
\,.
\end{equation}
If $\Oo_K$ is norm-Euclidean, then every pair of $\alpha, \beta \in \Oo_K$ has a gcd $\delta \in \Oo_K$ that can be determined, in principle, by the Euclidean algorithm. As explained in a first algebra course, it follows that $\Oo_K$ is a principal ideal domain and hence also a unique factorization domain.

Here we consider orders in number fields, not just the ring of integers.  Moreover, we work with a more general notion, first
introduced by Johnson--Queen--Sevilla (for rings of integers in quadratic fields) \cite{JQS85}.
We call an order $\Oo$ \textsf{generalized Euclidean} if for all $\alpha,\beta\in\Oo$ with $\beta\neq 0$ and $(\alpha,\beta)$ principal, there exists $\gamma, \rho\in \Oo$ satisfying (\ref{E:euclid}).
In other words, if the gcd of $\alpha$ and $\beta$ exists (in a strong form), then it can be determined (in principle) using the Euclidean algorithm (relative to the norm).
(To align more with current practice, we should perhaps use the terminology 
``generalized norm-Euclidean'', but we have chosen to use the shorter ``generalized Euclidean''.)
A partial classification of quadratic rings of integers that are generalized Euclidean was given in~\cite[Theorems 1 and 2]{JQS85},
but they did not fully treat the case of $\Q(\sqrt{d})$ when $d\equiv 1\pmod{4}$.
We prove a finiteness theorem for general quadratic orders.

\begin{thm}\label{thm:mainquadratic} There are finitely many quadratic orders that are generalized Euclidean.
\end{thm}

One might wonder whether we gain anything by considering nonmaximal orders in Theorem \ref{thm:mainquadratic}. Nonmaximal orders can never be Euclidean, as they are never unique factorization domains. However, it is possible for a nonmaximal order to be generalized Euclidean.  This is the case, for instance, for $\Z[\sqrt{-3}]$, which is nonmaximal with respect to the field $K:=\Q(\sqrt{-3})$. 

Lacking a reference, we prove here that $\Z[\sqrt{-3}]$ is generalized Euclidean. Let $\alpha, \beta \in \Z[\sqrt{-3}]$ with $\beta \ne 0$, and write $\alpha/\beta = U + V\sqrt{-3}$ with $U, V \in \Q$. Let $u, v$ be the integers nearest to $U, V$, rounding down when the distance to the nearest integer is precisely $\frac{1}{2}$. Let $\gamma = u+v\sqrt{-3}$ and define $\rho$ by \eqref{E:euclid}. Then $N_{K/\Q}(\rho) = N_{K/\Q}(\beta) ((U-u)^2 + 3(V-v)^2) < N_{K/\Q}(\beta)$, unless $U= u + \frac12$ and $V = v + \frac12$.  (We omit absolute value signs around our norms, since norms from $\Q(\sqrt{-3})$ are nonnegative.) So to show $\Z[\sqrt{-3}]$ is generalized Euclidean, it is enough to show that  $I:=(\alpha, \beta)$ is nonprincipal in this exceptional case. When this case arises, $2\alpha = 2\beta \gamma + (1+\sqrt{-3}) \beta$, so that
\[ 2I = (2\alpha,2\beta) = (2\beta,(1+\sqrt{-3})\beta) = \beta (2,1+\sqrt{-3}). \] 
Suppose for a contradiction that $I$ is principal, say $I = (\delta)$. Then $2I = (2\delta)$, $\beta \mid 2\delta$, and $J:=(2,1+\sqrt{-3}) = (2\delta/\beta)$ is principal. But it is straightforward to check that $J$ is nonprincipal. For instance, $2$ and $1+\sqrt{-3}$ are nonassociate irreducibles in $\Z[\sqrt{-3}]$, so that $J$ could not be generated by anything other than a unit. But every element of $J$ has even norm, and so $1\notin J$.

We may also consider more general orders.
Fix an integer $d\ge 2$. If $m$ is an integer for which the polynomial $F_{d,m}(T):=T^d-m$ is irreducible over $\Q$, we let $K_{d,m} = \Q(m^{1/d})$ and $\Oo_{d,m}:= \Z[m^{1/d}]$. Then $\Oo_{d,m}$ is an order inside the number field $K_{d,m}$; we will refer to $\Oo_{d,m}$ as a \textsf{pure order of degree~$d$}.

\begin{thm}\label{thm:maingeneral} Fix an even positive integer $d$. There are finitely many pure orders of degree $d$ that are generalized Euclidean.
\end{thm}

\begin{thm}\label{thm:maingeneralGRH} Assume GRH.  For any fixed integer $d\ge 2$, there are finitely many pure orders of degree $d$ that are generalized Euclidean.
\end{thm}

One might be interested in comparing Theorems \ref{thm:maingeneral} and \ref{thm:maingeneralGRH} with results of Egami (see~\cite{egami84}) for the fields $\Q(m^{1/p})$, $p$ prime. However, the notion of $\Q(m^{1/p})$ containing a norm-Euclidean ideal class, considered by Egami, is different than the notion we consider here, so neither set of results implies the other.

\section{Preparation}
At the heart of all our proofs are the following variants of \textsf{Heilbronn's criterion} for non-Euclideanity. 
Heilbronn's criterion was first introduced in \cite{heilbronn38}, but can be considered implicit in earlier work of Erd\H{o}s and Ko \cite{EK38}. We remark that it is condition (i) in both Propositions \ref{prop:pureorderheilbronn}  and \ref{prop:quadraticheilbronn} that allows the stronger conclusion of ``not generalized Euclidean'' versus ``not norm-Euclidean''.

\begin{prop}[Heilbronn's criterion for pure orders]\label{prop:pureorderheilbronn}
Suppose $d\ge 2$.  Let $\Oo_{d,m}$ be a pure order of degree $d$. Suppose there are positive integers $a,b$ with 
\[ |m| = a+b,\]
where 
\begin{enumerate}
\item $\gcd(a,m)=1$,
\item $a$ is a $d$th power modulo $m$,
\item neither $a$ nor $-b$ has the form $N_{K_{d,m}/\Q}(\alpha)$ for any $\alpha \in \Oo_{d,m}$. 
\end{enumerate}
Then $\Oo_{d,m}$ is not generalized Euclidean.
\end{prop}

Recall that to each quadratic order $\Oo$ there is associated a unique integer $D$ for which either
\begin{equation}\label{eq:firstform} \Oo = \{a+b\sqrt{D}: a, b \in \Z\} \end{equation}
or
\begin{equation}\label{eq:secondform} \Oo = \bigg\{\frac{1}{2}(a+b\sqrt{D}): a, b \in \Z, a \equiv b\pmod{2}\bigg\}. \end{equation}
The first case occurs when the order discriminant $\Delta$ is a multiple of $4$, in which case $D = \Delta/4$, and the second case when $\Delta\equiv 1\pmod{4}$, in which case $D=\Delta$. In both cases, $D$ is not a perfect square.

\begin{prop}[Heilbronn's criterion for quadratic orders]\label{prop:quadraticheilbronn} Let $\Oo$ be an order in a quadratic field $K$, and let $D\in\Z$ be such that \eqref{eq:firstform} or \eqref{eq:secondform} holds. Suppose there are positive integers $a,b$ with 
\[ |D| = a+b,\]
where 
\begin{enumerate}
\item $\gcd(a,D)=1$,
\item $a$ is a square modulo $D$,
\item neither $a$ nor $-b$ has the form $N_{\Q(\sqrt{D})/\Q}(\alpha)$ for any $\alpha \in \Oo$. 
\end{enumerate}
Then $\Oo$ is not generalized Euclidean.
\end{prop}
\noindent Proposition \ref{prop:quadraticheilbronn} does not quite follow from Proposition \ref{prop:pureorderheilbronn}; when $D\equiv 1\pmod{4}$, the corresponding order $\mathcal{O}$ is not $\Oo_{2,m}$ for any $m$.

Perhaps the chief novelty in the proofs (and statements) of Propositions \ref{prop:pureorderheilbronn} and \ref{prop:quadraticheilbronn} is the lack of any (explicit) reference to ramification. Instead, we phrase the arguments in terms of properties of the corresponding norm forms.

We suspect the following lemma, needed for the proof of Proposition \ref{prop:pureorderheilbronn}, is known to the experts. But lacking a suitable reference, we include the proof here.

\begin{lem}\label{lem:normform} Let $d\ge 2$, and let $\Oo_{d,m}$ be a pure order of degree $d$. Let $\tilde{N}(x_0,\dots,x_{d-1})$ be the \textsf{norm form} associated to the $\Q$-basis $1, m^{1/d}, \dots, m^{(d-1)/d}$ of $K_{d,m}$. That is, $\tilde{N}(x_0,\dots,x_{d-1})$ is the polynomial in $x_0,\dots,x_{d-1}$ defined by
\[ \tilde{N}(x_0,x_1,\dots,x_{d-1}) = \prod_{\sigma\colon K_{d,m}\hookrightarrow \C}(x_0 + x_1 \sigma(m^{1/d}) + \dots + x_{d-1} \sigma(m^{(d-1)/d})), \]
where $\sigma$ runs over all the embeddings of $K_{d,m}$ into $\C$. Then $\tilde{N}(x_0,\dots,x_{d-1}) \in x_0^{d} + m \Z[x_0,\dots,x_{d-1}]$. 
\end{lem}

\begin{proof} Define a polynomial $\tilde{N}(\mathbf{x},\mathbf{t})$ in indeterminates $\mathbf{x}=x_0,\dots,x_{d-1}$ and $\mathbf{t}=t_0,\dots,t_{d-1}$ by setting
\[ \tilde{N}(\mathbf{x},\mathbf{t}):= \prod_{j=0}^{d-1} (x_0 + x_1 t_j + x_2 t_j^2+ \dots + x_{d-1} t_j^{d-1}). \]
Then $\tilde{N}(\mathbf{x},\mathbf{t})$ can  be viewed as a polynomial in $\mathbf{t}$ over $\Z[\mathbf{x}]$, symmetric with respect to the $t_j$. By the fundamental theorem of symmetric polynomials, $\tilde{N}(\mathbf{x},\mathbf{t})$ is a polynomial, with $\Z[\mathbf{x}]$ coefficients, in the elementary symmetric functions of the $t_j$.

Continuing, we introduce a new indeterminate $Z$, and we let $\omega_j := \exp(2\pi \mathrm{i} j/d)$. Replacing each $t_j$ by $\omega_j Z$, we conclude from our work in the last paragraph that 
\begin{equation}\label{eq:normZversion} \tilde{N}(\mathbf{x},Z):= \prod_{j=0}^{d-1} (x_0 + x_1 \omega_j Z + x_2 \omega_j^2 Z^2 + \dots + x_{d-1} {\omega_j}^{d-1} Z^{d-1})\end{equation}
is a polynomial over $\Z[\mathbf{x}]$ in the elementary symmetric functions of the $\omega_j Z$. These elementary symmetric functions all vanish, except for the last one (the $d$th), which is $\pm Z^d$; this can be read off from the formal expansion $\prod_{j=0}^{d-1} (T - \omega_j Z) = T^d - Z^d$. We conclude that $\tilde{N}(\mathbf{x},Z)$ is a polynomial in $x_0,\dots,x_{d-1}$ and $Z^d$, with integer coefficients. Setting $Z=0$ in \eqref{eq:normZversion}, we deduce that 
\[ \tilde{N}(\mathbf{x}, Z) = x_0^{d} + Z^d Q \quad\text{for some}\quad Q\in \Z[\mathbf{x},Z^d].\]
To finish, it remains only to observe that
\[ \tilde{N}(x_0,x_1,\dots,x_{d-1}) = \tilde{N}(\mathbf{x}, m^{1/d}), \]
so that 
\[ \tilde{N}(x_0,x_1,\dots,x_{d-1}) \in x_0^d + m \Z[x_0,\dots,x_{d-1}],\] as claimed.
\end{proof}

\begin{proof}[Proof of Proposition \ref{prop:pureorderheilbronn}] It is convenient to abbreviate $\Oo_{d,m}$ to $\Oo$ throughout the proof, and to write $N$ in place of $N_{K_{d,m}/\Q}$. Suppose for a contradiction that $\Oo$ is generalized Euclidean. Using conditions (i) and (ii), we may choose $A \in \Z$ with $A$ coprime to $m$ and $A^d\equiv a \bmod{m\Z}$. Since
\[ A\Oo + m^{1/d} \Oo \supset A\Z + m\Z = \Z, \]
the $\Oo$-ideal generated by $A$ and $m^{1/d}$ is all of $\Oo$, and in particular is principal. As $\Oo$ is generalized Euclidean, there is an $\alpha \in \Oo$ with $A\equiv \alpha \pmod{m^{1/d}\Oo}$ and $|N(\alpha)| < |N(m^{1/d})| = |m|$. Write
\[ \alpha = A_0 + A_1 m^{1/d} + \dots +A_{d-1} m^{(d-1)/d}.\]
Then $A\equiv \alpha \equiv A_0 \pmod{m^{1/d}\Oo}$, so that
$(A-A_0)/m^{1/d} \in \Oo$.
But 
\[ \frac{A-A_0}{m^{1/d}} = \frac{A-A_0}{m} m^{(d-1)/d}; \]
as $\Oo = \Z[m^{1/d}]$, for the final displayed expression to land in $\Oo$ requires that $(A-A_0)/m \in \Z$. Hence, $A_0 \equiv A\pmod{m}$. By Lemma \ref{lem:normform},
\begin{align*} N(\alpha) &= \tilde{N}(A_0,\dots,A_{d-1}) \\
&\equiv A_0^d \equiv A^d \equiv a\pmod{m\Z}. \end{align*}

We have $0 < a < |m|$ and $|N(\alpha)| < |m|$. If $0 \le N(\alpha) < |m|$, then the congruence $N(\alpha)\equiv a\pmod{m\Z}$ forces $a = N(\alpha)$. Similarly, if $-|m| < N(\alpha)\le 0$, then $-b=a-|m| = N(\alpha)$. In either case we contradict our condition (iii).
\end{proof}

\begin{proof}[Proof of Proposition \ref{prop:quadraticheilbronn}] This is very similar to the proof of Proposition \ref{prop:pureorderheilbronn}, so we only sketch it. Using (i) and (ii), choose $A \in \Z$, $\gcd(A,D)=1$, with $A^2\equiv a\pmod{D}$. Then $A\Oo+\sqrt{D}\Oo \supset A\Z + D\Z=\Z$. Hence, $A, \sqrt{D}$ generate the unit ideal of $\Oo$. Thus, assuming $\Oo$ is generalized Euclidean, there is an $\alpha \in \Oo$ with $A\equiv \alpha\pmod{\sqrt{D}\Oo}$ and $|N(\alpha)| < |N(\sqrt{D})| = |D|$, where $N$ denotes the norm $N_{\Q(\sqrt{D})/\Q}$. 

Let $\sigma$ denote the nontrivial automorphism of $\Q(\sqrt{D})$. Since $\sigma$ carries $\sqrt{D}\Oo$ into $\sqrt{D}\Oo$, the congruence \[ A\equiv \alpha\pmod{\sqrt{D}\Oo} \] implies \[ A\equiv \sigma(\alpha)\pmod{\sqrt{D}\Oo}. \] Thus, $A^2 \equiv \alpha\sigma(\alpha)\equiv N(\alpha)\pmod{\sqrt{D}\Oo}$, and \[ 
\frac{A^2-N(\alpha)}{\sqrt{D}} = \frac{A^2-N(\alpha)}{D} \sqrt{D} \in \Oo. \]
Regardless of whether we are in case \eqref{eq:firstform} or \eqref{eq:secondform}, this last containment forces $(A^2-N(\alpha))/D \in \Z$. Hence,
\[ a \equiv A^2 \equiv N(\alpha)\pmod{D\Z}. \]
Reasoning as in Proposition \ref{prop:pureorderheilbronn}, we deduce from $|N(\alpha)| < |D|$ that either $N(\alpha)=a$ or $N(\alpha)=-b$, contradicting (iii).
\end{proof}
 
\section{Proofs of Theorems \ref{thm:mainquadratic}--\ref{thm:maingeneralGRH}}

Theorems \ref{thm:mainquadratic}--\ref{thm:maingeneralGRH} follow immediately from Propositions \ref{prop:pureorderheilbronn} and \ref{prop:quadraticheilbronn} via the following lemmata.

\begin{lem}\label{lem:generalevend}
Let $\Oo_{d,m}$ be a pure order of even degree $d\geq 2$.
Provided $|m|$ is sufficiently large, there are positive integers $a, b$ with $|m|=a+b$ and conditions (i)--(iii) of Proposition \ref{prop:pureorderheilbronn} satisfied.
\end{lem}

\begin{lem}\label{lem:quadraticrep}
Let $\Oo$ be a quadratic order with $D\in\Z$ satisfying
\eqref{eq:firstform} or~\eqref{eq:secondform}.
Provided $|D|$ is sufficiently large, there are positive integers $a, b$ with $|D|=a+b$ and conditions (i)--(iii) of Proposition \ref{prop:quadraticheilbronn} satisfied.
\end{lem}

\begin{lem}\label{lem:generalGRH}
    Assume GRH.
    Let \( d \ge 2 \). For every $m \in \Z$ with $|m|$ sufficiently large, there are positive integers $a, b$ with $|m|=a+b$ and conditions (i)--(iii) of Proposition \ref{prop:pureorderheilbronn} satisfied.
\end{lem}

Our proofs of Lemmas \ref{lem:generalevend} and~\ref{lem:quadraticrep} rely on the following result of Pollack \cite[Theorem 1.1]{pollack17}.

\begin{prop}\label{prop:smallnonresidues} For each $\epsilon > 0$, there are $m_0 = m_0(\epsilon)$ and $\kappa=\kappa(\epsilon)>0$ for which the following holds. For all $m > m_0$ and every nontrivial Dirichlet character $\chi\bmod{m}$, there are more than $m^{\kappa}$ primes $q \le m^{\frac{1}{4\sqrt{e}} + \epsilon}$ for which $\chi(q) \notin \{0,1\}$.
\end{prop}

\begin{proof}[Proofs of Lemmas~\ref{lem:generalevend} and \ref{lem:quadraticrep}] In order to treat both Lemmas in a unified fashion, we adopt the following convention. When Lemma \ref{lem:generalevend} is in view, the letters $d,m$ have the meanings of that lemma, $K=K_{d,m}$ and $\mathcal{O}=\mathcal{O}_{d,m}$. If we are considering Lemma \ref{lem:quadraticrep}, we take $d=2$, $m=D$, $K= \Q(\sqrt{D})$ and $\mathcal{O}$ the quadratic order appearing in that lemma.

To start with, we produce positive integers $a$ and $b$ summing to $|m|$ and satisfying conditions (i) and (ii) in Proposition \ref{prop:pureorderheilbronn} (respectively, Proposition \ref{prop:quadraticheilbronn}). At the end of the proof we will see that our particular construction of $a$ and $b$ also satisfies condition (iii) of these Propositions. 

%

Let $\chi$ be the quadratic character mod $4|m|$ given by $\chi(n) = (\frac{4m}{n})$ (the Kronecker symbol),
and note that if $\chi(q)\not\in\{0,1\}$ for the prime $q$, then $q$ is odd and $(\frac{m}{q})=-1$.
By Proposition \ref{prop:smallnonresidues}, it follows that for all \( \epsilon > 0 \), there is a $\kappa > 0$ with
\[ \# \left\{ q \le |m|^{\frac{1}{4\sqrt{\mathrm{e}}} + \epsilon} : q~\text{odd prime}, \left(\frac{m}{q}\right) = -1  \right\} \ge |m|^{\kappa} \]
whenever \( m \) is sufficiently large in absolute value. Since \( 4\sqrt{\mathrm{e}} > 6 \), we can take, for \( m \) large in absolute value, \( 2 < q_1 < q_2 \le |m|^{1/6} \) with \( (\frac{m}{q_1}) = (\frac{m}{q_2}) = -1 \). (In particular, \( q_1, q_2 \nmid m\).)

We now proceed to count the number of integers \( a', b' \) for which
\begin{equation}\label{eq:sumrep} q_1 a' + q_2 b' = |m|\end{equation}
and for which all of the following hold:
\begin{enumerate}
    \item[1.] \( 0 < a' < \dfrac{|m|}{q_1} \);
    \item[2.] \( q_1a' \) is a $d$th power mod \( m \) with $\gcd(q_1 a',m)=1$;
    \item[3.] \( q_1 \nmid a' \);
    \item[4.] \( q_2 \nmid b' \).
\end{enumerate}
If at least one such pair of \( a', b' \) exist, then conditions (i) and (ii) in Proposition \ref{prop:pureorderheilbronn} (or Proposition \ref{prop:quadraticheilbronn}) are satisfied by taking \( a = q_1a', b = q_2b'. \)

Equation \eqref{eq:sumrep} implies that $b' = (|m|-q_1a')/q_2$ is entirely determined by $a'$. So we can phrase the problem entirely in terms of counting certain integers $a'$. To ensure $b'$ is an integer, we tack on the condition 
\begin{enumerate}
\item[5.] \( a' \equiv |m|q_1^{-1} \pmod{q_2} \).
\end{enumerate}
Our task then becomes counting all the integers $a'$ simultaneously satisfying conditions 1--5, thinking of condition 4 as \( a' \not\equiv |m|q_1^{-1} \pmod{q_2^2} \). 

We detect condition 2 using character sums; we have that
\[ \frac{1}{N}\sum_{\chi^d = \chi_0} \chi(a) = \begin{cases}
    1 & \text{if }a \text{ is a $d$th power mod }m \text{ and }\gcd(a, m) = 1; \\ 0 &\text{otherwise,}
\end{cases} \]
where \( N \) is the number of characters \( \chi \) mod \( m \) satisfying \( \chi^d = \chi_0 \).

We use an inclusion-exclusion argument to account for conditions 3 and 4. We thus write our count as
\[ \frac{1}{N}\sum_{0 \le i, j \le 1} (-1)^{i+j}{\sum_{a'}}^{(i,j)}\sum_{\chi^d = \chi_0}\chi(q_1a') \]
where the sums \( {\sum_{a'}}^{(i,j)} \) are taken over all \( a' \) satisfying conditions 1 and 5, but possibly violating conditions 3 and/or 4: here \( i = 1 \) if (and only if) condition 3 is violated, and \( j = 1 \) if (and only if) 4 is violated.

We proceed to estimate the sums \( {\sum_{a'}}^{(i,j)}\sum_{\chi^d = \chi_0}\chi(q_1a') \). First, we look at the contribution of the nonprincipal characters. A violation of condition 3 means that we have \( a' \equiv 0 \pmod{q_1} \). Similarly, a violation of condition 4 means that condition 5 can be strengthened to \( a' \equiv |m|q_1^{-1} \pmod{q_2^2} \). Thus, depending on \( i \) and \( j \), we have by the Chinese Remainder Theorem that each \( a' \) lies in one specific residue class mod \( m_{i,j} := q_1^{i}q_2^{j+1} \), i.e., we can write \( a' = a_{i,j} + nm_{i,j} \) where \( n \) lies in an interval \( I_{i,j} \) of length \( |m|/(q_1^{i+1}q_2^{j+1}) \). We then write, for each \( \chi \ne \chi_0 \),
\begin{align*}
    {\sum_{a'}}^{(i,j)} \chi(q_1a') &= \chi(q_1){\sum_{a'}}^{(i,j)} \chi(a') \\
    &= \chi(q_1)\sum_{n \in I_{i,j}}\chi(nm_{i,j} + a_{i,j}) \\
    &= \chi(q_1)\chi(m_{i,j})\sum_{n\in I_{i,j}}\chi(n + a_{i,j}m_{i,j}^{-1}) \\
    &= O\left(\sqrt{|m|}\log{|m|}\right)
\end{align*}
by P\'{o}lya--Vinogradov. Here we write \( m_{i,j}^{-1} \) for the inverse of \( m_{i,j} \) mod \( m \).

We now look at the contribution of the principal character. We write
\begin{align*}
    {\sum_{a'}}^{(i,j)} \chi_0(q_1a') &= {\sum_{a'}}^{(i,j)} \chi_0(a) \\
    &= {\sum_{a'}}^{(i,j)} \mathbf{1}_{\gcd(a',m) = 1} \\
    &= {\sum_{a'}}^{(i,j)} \sum_{e \mid \gcd(a',m)} \mu(e) \\
    &= \sum_{e \mid m}\mu(e){\sum_{a'}}^{(i,j)} \mathbf{1}_{e \mid a'} \\
    &= \sum_{e \mid m}\mu(e)\left(\frac{|m|}{eq_1^{i+1}q_2^{j+1}} + O(1)\right) \\
    &= \frac{\varphi(|m|)}{q_1^{i+1}q_2^{j+1}} + O\left(2^{\omega(|m|)}\right).
\end{align*}

Collecting estimates,
\[ {\sum_{a'}}^{(i,j)}\sum_{\chi^d = \chi_0}\chi(q_1a') = \frac{\varphi(|m|)}{q_1^{i+1}q_2^{j+1}} + O\left(2^{\omega(|m|)}\right) + O\left(N\sqrt{|m|}\log|m|\right). \]

Putting everything together, we get that, as \( |m| \to \infty \), the number of integer solutions \( a, b \) to \eqref{eq:sumrep} satisfying conditions 1--5 is
\[ \frac{\varphi(|m|)}{Nq_1q_2}\left(1 - \frac{1}{q_1}\right)\left(1 - \frac{1}{q_2}\right) + O\left(\frac{2^{\omega(|m|)}}{N}\right) + O\left(\sqrt{|m|}\log{|m|}\right). \]

Let us estimate the various terms in this expression. Writing $\tau$ for the divisor function, we have $2^{\omega(|m|)} \le \tau(|m|) \le |m|^{o(1)}$ (see Theorem 315 in \cite{HW08} for this last bound on $\tau$). Since the multiplicative group mod $m$ has a direct product decomposition into at most $\omega(|m|)+1$ cyclic groups (see for instance \cite[Theorem 3, p.\ 44]{IR90}), \[ N \le d \cdot d^{\omega(|m|)} \le d \cdot 2^{d \omega(|m|)} \le d \cdot \tau(|m|)^{d} \le |m|^{o(1)}. \] Furthermore, Theorem 327 in \cite{HW08} ensures that for sufficiently large $|m|$, we have 
\( \varphi(|m|) \ge |m|^{0.99} \). Putting these estimates together, keeping in mind that $q_1, q_2 \le |m|^{1/6}$, we see that the main term in the last display exceeds $m^{0.6}$ for large $|m|$, while the contribution of the $O$-terms is smaller than $m^{0.51}$ (say). In particular, the displayed expression is positive for all large $|m|$, allowing us to find our $a', b'$.


Finally we return to proving our choices of $a$ and $b$ satisfy part (iii) of Propositions \ref{prop:pureorderheilbronn} and \ref{prop:quadraticheilbronn}. Suppose that $a = N(\alpha)$, with $\alpha \in \Oo$. We consider the prime factorization of $\alpha\Oo_K$ in the Dedekind domain $\Oo_K$. By our condition 3 above, \( q_1 \parallel a \). Thus, the prime factorization of $\alpha\Oo_K$ contains a unique prime ideal $Q_1$ lying above $q_1$; furthermore, $Q_1$ has degree one.  Then \(  \Oo_K/Q_1 \cong \Z/q_1 \Z\), with the isomorphism induced by the inclusion of \( \Z \) into \( \Oo_K \). Since \( m = (m^{1/d})^d \) is a \( d \)th power in \( \Oo_K \), it is also a $d$th power in \( \Oo_K/Q_1 \) and thus in \( \Z/q_1\Z \). But $d$ is even, and so the stipulation that $m$ is a quadratic nonresidue modulo \( q_1 \) implies that $m$ is not a $d$th power modulo $q_1$. This contradiction establishes that $a$ is not a norm in \( \Oo \). An entirely analogous argument shows that $-b$ is not a norm in \( \Oo \). Here we use that $q_2 \parallel b$ (by condition 4 above) and that $m$ is a quadratic nonresidue modulo $q_2$. 
\end{proof}

We now settle Lemma \ref{lem:generalGRH}. In view of Lemma \ref{lem:generalevend}, we may assume $d$ is odd. We make use of the following GRH-conditional version of the Chebotarev density theorem:
\begin{prop}\label{chebotarev} Assume GRH. Let \( E/\Q \) be a finite Galois extension of number fields, and let \( G \) be the Galois group of \( E/ \Q\). Let \( C \) be a subset of \( G \) that is invariant under conjugation. Let \( \pi_C(x) \) denote the number of primes up to \( x \) whose associated Frobenius conjugacy class is in \( C \). Let \( P\) be the product of the distinct rational primes ramifying in \( E \). Then 
\[ \left|\pi_C(x) - \frac{|C|}{|G|}\Li(x)\right| \le A |C| x^{1/2}\log{\left([E : \Q] P x \right)}  \]
for some absolute constant \( A \).
\end{prop}

Proposition \ref{chebotarev} is the special case of Serre's \cite[eq.\, ($20_{R}$)]{MR644559} where $K=\Q$.

\begin{proof}[Proof of Lemma \ref{lem:generalGRH}] Let $d > 1$ be odd, and let \( p > 2 \) be the least prime dividing \( d \). Arguing as in the proof of Lemma \ref{lem:generalevend}, it is enough to show that there are two primes $q_1< q_2 \le |m|^{1/6}$, not dividing $m$, modulo which $m$ is not a $p$th power.

We let \( \zeta_p = \mathrm{e}^{2\pi i/p}  \) and we consider \( E := \Q(m^{1/p}, \zeta_p) \), the splitting field of the polynomial \( T^p - m \). As $E$ is the normal closure of $K:=\Q(m^{1/p})$, every prime ramifying in $E$ divides $\disc(K)$, which in turn divides $\disc(T^p-m) = (-1)^{(p-1)/2}p^p m^{p-1}$. Hence: Every prime ramifying in $E$ divides $mp$. This observation will be useful momentarily when we apply Proposition \ref{chebotarev}.

Let $q$ be a prime not dividing $mp$. In order that $m$ not be a $p$th power modulo $q$, it is sufficient that the action of the Frobenius conjugacy class associated to $q$, on the roots of $T^p-m$, have no fixed points. We let \( C \) be the (conjugation-invariant) set of fixed-point-free elements of $G$.

The Galois group \( G \) of \( E/\Q \) is made up of $p(p-1)$ elements \( \sigma_{j,k} \) for \( 0 \le j \le p-1 \), \( 1 \le k \le p-1 \), determined by the conditions that \( \sigma_{j,k}(m^{1/p}) := m^{1/p}\zeta_p^j \), \( \sigma_{j,k}(\zeta_p) := \zeta_p^k \). Furthermore, it is straightforward to check that the set $C$ defined above is precisely the collection of \( \sigma_{j,1} \) with \( j \ne 0 \). Hence, \( |C|/|G| = 1/p \). 

Applying Proposition \ref{chebotarev}, we conclude that $|\pi_C(x) - \frac{1}{p}\Li(x)| \le A_d x^{1/2}\log|pmx|$
for some constant \( A_d \) depending only on \( d \). It follows that
\begin{equation}\label{eq:usedcheb} \pi_C(x) \ge \frac{1}{p}\Li(x) - A_dx^{1/2}\log|pmx|. \end{equation}

It is certainly possible to choose our two primes \( q_1, q_2 \) as long as \begin{equation}\label{eq:crucialinequality} \pi_C(|m|^{1/6}) \ge 2 + \omega(p|m|). \end{equation} We will show this inequality holds for $x = (\log|m|)^3$, provided $m$ is large enough. As $(\log|m|)^3 \le |m|^{1/6}$ for large $|m|$, this suffices.

Allowing implied constants to depend on $d$ (which is fixed), we have \( \Li(x) \gg \log^3|m|/\log\log|m| \), while \( x^{1/2}\log|pmx| \ll \log^{5/2}|m| \). Hence, \eqref{eq:usedcheb} yields \( \pi_C(x) \gg \log^3|m|/\log\log|m| \) for large $|m|$. On the other hand, \( \omega(p|m|) \ll \log|m|, \) which is of smaller order. Thus \eqref{eq:crucialinequality} holds whenever $|m|$ is sufficiently large, as desired.
\end{proof}

\subsection*{Acknowledgements}
We thank the referee for a careful reading and helpful suggestions.

\bibliography{citation}
\bibliographystyle{amsalpha}
\end{document}